\documentclass[microtype]{gtpart}     
\agtart  
\usepackage{mathrsfs}
\usepackage{enumerate}
\usepackage{latexsym,amsxtra}
\usepackage{dsfont}
\usepackage{slashed}
\usepackage[all]{xy}
\usepackage{amsmath,amsfonts,amsthm,amssymb}
\usepackage{latexsym,amsmath}
\usepackage{graphicx,psfrag}
\usepackage{mathabx}
\usepackage{hyperref}
\usepackage{fancyhdr}
\usepackage{slashed}
\usepackage{pinlabel}
\usepackage[dvipsnames]{xcolor}

\newcommand{\F}{\mathbb F}

\newcommand{\sign}{\operatorname{sign}}

\newcommand{\SU}{\operatorname{SU}}

\newcommand{\U}{\operatorname{U}}

\newcommand{\FO}{\operatorname{FO}}

\newcommand{\Int}{\operatorname{int}}

\renewcommand{\phi}{\varphi}

\renewcommand{\sf}{\operatorname{sf}}

\newcommand{\coker}{\operatorname{coker}}

\renewcommand{\sf}{\operatorname{sf}}

\newcommand{\lfo}{\lambda_{\,\FO}}

\newcommand{\ts}{T^2 \times S^2}
\newcommand{\sss}{S^1 \times S^3}
\newcommand{\cptwo}{\C P^2}

\newcommand{\pp}{\pi} 

\newcommand{\la}{\langle}
\newcommand{\ra}{\rangle}

\newcommand{\jtilde}{\widetilde{J}}

\newcommand{\ftilde}{\widetilde{F}}

\newtheorem{thm}{Theorem}[section]

\newtheorem*{theo}{Theorem}

\newtheorem{lemma}[thm]{Lemma}
\newtheorem{corollary}[thm]{Corollary}
\newtheorem{proposition}[thm]{Proposition}

\newtheorem{introthm}{Theorem}

\theoremstyle{definition}
\newtheorem{definition}[thm]{Definition}
\newtheorem{example}[thm]{Example}

\newcommand{\TL}{Levine-Tristram}

\title{A Levine-Tristram invariant for knotted tori}
\author{Daniel Ruberman}
\givenname{Daniel}
\surname{Ruberman}
\address{Department of Mathematics, MS 050\newline\indent Brandeis
University \newline\indent Waltham, MA 02454}
\email{{ruberman@brandeis.edu}}
\arxivreference{2010.02355}

\begin{document}
\begin{abstract}
We define a new topological invariant of an embedded torus in a homology $S^1\times S^3$, analogous to the \TL\  invariant of a knot. We compare it to an invariant of smooth tori, defined recently by Echeverria using gauge theory for singular connections.
\end{abstract}
\maketitle

\section{Introduction}
The \TL\  invariant, $\sigma_\alpha(K)$, of an odd-dimensional knot~\cite{levine:cobordism,tristram} (see also Conway's survey~\cite{conway:TL}) is a function from the unit circle to the integers.  We will use this invariant in dimension $3$, where it is defined for a knot $K$ in an oriented homology $3$-sphere. By definition, $\sigma_\alpha(Y, K)$ is the signature of a certain Hermitian form (constructed from the Seifert form) that depends on $\alpha \in S^1$.  The \TL\  invariant plays a central role in the calculation of the knot concordance groups in higher dimensions, and has interesting connections to gauge theory in dimension $3$.  

In this paper we construct, by topological means, a $4$-dimensional version of the \TL\ invariant for certain knotted tori, and discuss its relationship to a gauge-theoretic invariant of embedded tori introduced by Echeverria~\cite{echeverria:tori}. To state our result, let us define a homology  $S^1 \times S^3$ to be an oriented $4$-manifold $X$ with the homology of $S^1 \times S^3$. A homology orientation of X (usually suppressed in the notation, but still lurking in the background) is a choice, $\gamma$, of generator of the first cohomology $H^1(X) \cong \Z$. 
\begin{introthm}\label{T:TL}
Let $T$ be a locally flat torus in a homology oriented homology $S^1 \times S^3$ so that the inclusion map induces a surjection on the first homology groups. Then for any $\alpha \in S^1$, there is an invariant $\sigma_\alpha(X,T)\in \Z$ with the following properties.
\begin{enumerate}
\item If $K$ is a knot in a homology 3-sphere $Y$, then $\sigma_\alpha(S^1 \times Y,S^1 \times K) = \sigma_\alpha(Y,K)$.
\item Reversing either the orientation or the homology orientation of $X$ changes the sign of $\sigma_\alpha(X,T)$.
\end{enumerate}
\end{introthm}
The signature invariant $\sigma_\alpha(X,T)$ in theorem~\ref{T:TL} will be constructed as follows. From a torus as in Theorem~\ref{T:TL}, we will construct an oriented manifold $V$ with the cohomology of $\ts$, with a distinguished class in $H_3(V)$. If $\alpha$ has prime-power order, then $\sigma_\alpha(X,T)$ will be a standard Atiyah-Singer invariant of a $3$ manifold carrying that homology class together with a $\U(1)$ representation determined by $\alpha$. The function $\sigma_\alpha(X,T)$ will be extended to all $\alpha$ using a standard averaging argument. \\[2ex]
\textbf{Acknowledgments:}\; Thanks to Andrei Pajitnov and Alex Suciu for explaining their work on the homology of infinite cyclic covers, to Nikolai Saveliev and Langte Ma for comments on preliminary versions of the ideas in this paper, and to Anthony Conway, Stefan Friedl, and Chuck Livingston for some helpful correspondence. This paper was inspired by the paper of Mariano Echeverria described below, and I thank Mariano for explaining his work as it developed and for comments on an earlier draft of this paper. The  author was partially supported by NSF Grant DMS-1811111.

\subsection{Motivation and background}\label{S:back}
To motivate Theorem~\ref{T:TL} we briefly discuss the interpretation of $ \sigma_\alpha(Y, K)$ in terms of $3$-dimensional gauge theory. This comes from independent work of Herald~\cite{herald:alexander} and Heusener-Kroll~\cite{heusener-kroll}, building on an insightful paper of X.-S. Lin~\cite{lin:rep} and a suggestion of the author (see also~\cite{heusener:curves}). They identified $\frac12 \sigma_{\alpha^2}(K)$ with a (signed) count of flat $\SU(2)$ connections on the knot complement having holonomy conjugate to the matrix
\begin{equation}\label{E:meridian}
\begin{pmatrix}
\bar\alpha & 0\\
0 & \alpha
\end{pmatrix}.
\end{equation}
on the meridian of $K$. The signs are determined using spectral flow, as in Taubes' gauge-theoretic interpretation~\cite{taubes:casson} of the Casson invariant.  In Herald's version, the knot can lie in a homology sphere $Y$, and the Casson invariant~\cite{akbulut-mccarthy,saveliev:casson} of $Y$ enters into the formula.

A recent preprint by Echeverria~\cite{echeverria:tori} constructs an invariant $\lfo(X,T,\alpha)$ of knotted tori, similar to the count used in~\cite{herald:alexander,heusener-kroll}.  (Our notation is not quite the same as in~\cite{echeverria:tori}, where $\alpha$ denotes an element of $(0,1)$; in this paper $\alpha$ is the corresponding element of $S^1$.) The context, extending work of Furuta and Ohta~\cite{furuta-ohta}  is the following. Consider an oriented $4$-manifold $X$ with the homology of $S^1 \times S^3$, with a choice of generator of $H^1(X)$. One is given a torus $T^2$ embedded in $X$ so that the map $H_1(T) \to H_1(X)$ is surjective; let us call that an essential embedding of $T$. Then Echeverria shows how to count flat $\SU(2)$ connections on $X - T$ having specified holonomy \eqref{E:meridian} on the meridian $\mu_T$ of $T$, by interpreting such flat connections as elements of moduli spaces introduced in~\cite{kronheimer-mrowka:I,kronheimer-mrowka:II}.

There are some restrictions on the scope of Echeverria's invariant. One is a restriction on the topology of $X - T$, analogous to the condition introduced by Furuta and Ohta~\cite{furuta-ohta} in defining a Casson-type invariant  $\lfo$ associated to a homology $\sss$.
Moreover, $\lfo(X,T,\alpha)$ is defined only for $\alpha$ of finite order, in other words of the form  $\alpha=  e^{2 \pi i q/n}$.  The restriction stems from the use of orbifold instanton theory in the basic setup.  In the product case 
\[
(X,T) = S^1 \times (Y,K),
\]
Echeverria~\cite[Corollary 41]{echeverria:tori} shows that 
\begin{equation}\label{E:prod}
\lfo(X,T,\alpha) = 8 \lambda(Y) +  \sigma_{\alpha^2}(Y,K)
\end{equation}
where $\lambda(Y)$ is the Casson invariant~\cite{akbulut-mccarthy} of $Y$. 

Echeverria~\cite{echeverria:tori} posed the problem of finding a topological invariant $\sigma_\alpha(X,T)$, defined in a fashion similar to the \TL\  invariant, that satisfies a version of \eqref{E:prod} with 
$\lambda(Y)$ replaced by $\lfo(X)$ (when defined) and the knot signature replaced by 
$\sigma_\alpha(X,T)$. One might even hope to establish such a relation via methods similar to~\cite{herald:alexander,heusener-kroll}. Such a \TL\  invariant should have the property that in the product case we have $\sigma_\alpha(X,T) = \sigma_\alpha(Y,K)$. In this paper, we construct such a topological invariant $\sigma_\alpha(X,T)$.  

Subsequent to the posting of the first version of this paper, Langte Ma~\cite{ma:signature} has shown that our topological invariant is related to Echeverria's gauge-theoretic invariant $\lfo(X,T,\alpha)$ in the hoped-for fashion.
\begin{theo}[Ma 2021] When both invariants $\lfo(X,T,\alpha)$ and $\lfo(X)$ are defined, 
\begin{equation}\label{E:conj}
\lfo(X,T,\alpha) = 8\lfo(X) + \sigma_{\alpha^2}(X,T) 
\end{equation} 
\end{theo}

To state the main theorem, we make use of a special case of an invariant of Atiyah-Singer~\cite{atiyah-singer:III}. There is a good deal of variation in the terminology and notation for such invariants; we follow the version in~\cite{aps:II}.  Let $Y$ be an oriented $3$-manifold and choose a representation $\phi: \pi_1(Y) \to \U(1)$. Assume that there is an oriented $4$-manifold $W$ with boundary $W$ and that the representation $\phi$ extends to $\pi_1(W)$. Then define
\begin{equation}\label{E:rho}
\rho_\phi(Y) = \sign(W) - \sign_\phi(W)
\end{equation}
where $\sign_\phi(W)$ denotes the signature of the Hermitian intersection form on the twisted homology $H_2(W;\C_\phi)$. For example, if $Y$ is $0$-framed surgery on an oriented knot $K$ with meridian $\mu_K$ and $\phi_\alpha(\mu_K) = \alpha$, then $\rho_{\phi_\alpha}(Y)$ is~\cite{viro:links,kauffman-taylor:links} exactly the \TL\  invariant $ \sigma(K,\alpha)$.

Now we can outline the definition of our invariant $\sigma_\alpha(X,T)$, with the proof that it is well-defined discussed in the following sections. Let $T$ be an oriented embedded torus, as in Theorem~\ref{T:TL}.  One can perform~\cite[\S 2.3]{ma:surgery} the analogue of `$0$-framed surgery' on $X$ along this torus to obtain a manifold $V$ having the homology of $T^2 \times S^2$.  (This process is not canonical, but in Section~\ref{S:gluing} we will show that the resulting invariant is independent of all choices.) A simple but crucial observation is that in fact $V$ is a cohomology $\ts$, in other words has the same cohomology ring as $\ts$. A second observation is that the homology orientation $\gamma \in H^1(X)$ gives rise to a primitive cohomology class in $H^1(V)$; we will continue to call this class $\gamma$. 

Let $g \in H_1(T)$ be a homology class for which $\la \gamma, g \ra = 1$. Then $g$, together with the meridian $\mu = \mu_T$ give a basis for $H_1(V)$.  
For any $\alpha \in S^1$, consider the $\U(1)$ representation $\phi$ with $\phi(g) = 1$ and $\phi(\mu) = \alpha$. Choose a connected oriented $3$-manifold $M\subset V$ that is Poincar\'e dual to $\gamma$.  Then $\phi$ restricts to a $\U(1)$ representation of $\pi_1(M)$; we continue use the same notation for the restriction. Note that while the choice of $g$ is not canonical, the restriction of $\phi$ to $\pi_1(M)$ is still uniquely determined.

The definition then proceeds in two steps. First we show (Proposition~\ref{P:prime-power}) that if $\alpha$ has prime-power order, then $\rho_{\phi}(M)$ is independent of the choice of $3$-manifold dual to $\gamma$.  The second step will be to extend this definition to arbitrary characters $\phi$, using the fact that $\rho_{\phi}(M)$ is a piecewise constant function of $\alpha$, with finitely many jumps.  

As it turns out, it is not essential that $V$ come from surgery on a torus, so we start by defining invariants for a pair $(V,\gamma)$ consisting of an arbitrary cohomology $\ts$ and a primitive class $\gamma \in H^1(V)$.  We detail the properties needed in the next section.

\subsection{Notation and conventions}\label{S:notation} Throughout the paper, $V$ will denote an oriented cohomology $\ts$.  We fix a primitive element $\gamma \in H^1(V)$ and choose a second class $\eta$ so that $\{\gamma, \eta\}$ is a basis for $H^1(V)$.  Write  $\{g,h\}$ for the dual basis for $H_1(V)$. The condition on the cohomology of $V$ amounts to saying that $\gamma \cup \eta$ generates a summand of $H^2(V)$.  A complementary summand is generated by a class $\zeta \in H^2(V)$ so that $\zeta \cup \zeta = 0$ and $\gamma \cup \eta \cup \zeta$ gives the orientation of $V$. 

In the course of the proof, we will need to compare the cohomology of $V$ with that of the actual $\ts$.  To that end, we denote the corresponding cohomology classes in $\ts$ as follows.  The basis for $H^1(\ts)$ is written\ as $\{\xi,\beta\}$; we can assume that $\la \xi \cup \beta, [T]\ra = 1$, where $[T]$ is the homology class of the torus in $H_2(\ts)$. The distinguished element of $H^1(\ts)$ will be $\xi$. The other generator $\delta \in H^2(\ts)$ is the Poincar\'e dual of $[T]$.

We will denote by $A_x$ the infinite cyclic covering space of a space $A$ determined by a cohomology class $x\in H^1(A;\Z)$. Similarly, if $\beta \in H^1(A;\Z_d)$ then $A_x$ will denote the corresponding finite cyclic covering.  If $x \in  H^1(A;\Z)$ then $x_d$ will denote its reduction mod $d$, lying in $H^1(A;\Z_d)$. For the rest of the paper, $p$ will denote a prime, and $\F_p$ the field with $p$ elements.  

\section{Infinite cyclic coverings}\label{S:milnor}
Our approach is based on Milnor's Duality Theorem~\cite{milnor:covering} which states that the infinite cyclic cover of an $(n+1)$-manifold has, under some circumstances, the homological properties of a closed $n$-manifold. We will also make use of an extension of this principle~\cite{ruberman:ds,ruberman:ds2} that says that under appropriate hypotheses, the infinite cyclic cover of a $2n$-manifold can have $\rho$ invariants for a $\U(1)$ representation $\phi$ as if it were a closed $(2n-1)$-manifold.  The main point to establish is that the cyclic cover (corresponding to $\phi$) of the infinite cyclic cover also satisfies Poincar\'e duality. One might say that this extends Milnor's theorem to certain twisted coefficients. 

In this section, we prove such an extended duality theorem for $V$, a cohomology $\ts$. We need a preliminary lemma.
\begin{lemma}\label{L:map}
There is a degree-one map from $V$ to $\ts$ inducing isomorphisms on integral homology and cohomology, preserving the distinguished generators of the first cohomology group.
\end{lemma}
\begin{proof}
First choose maps $f_1$ and $f_2$ from $V$ to $S^1$ such that $f_1^*(\xi) = \gamma$ and $f_1^*(\beta) = \eta$. The product of these two gives a map $f$ from $V$ to $T^2$. We claim that there is a map $h: V \to S^2$ such that $h^*$ of the generator of $H^2(S^2)$ is $\zeta$. By standard obstruction theory, there is a map $h_1: V \to \cptwo$ such that $h_1^*(\delta') = \zeta$, where $\delta'$ is the generator of $H^2(\cptwo)$.
But the degree of $h_1$ must be $0$, as the following calculation shows.
\begin{equation*}
\begin{aligned}
\deg(h_1) =  \la \delta' \cup \delta' , \deg(h_1) [\cptwo] \ra & = \la \delta' \cup \delta' , (h_1)_*([V]) \ra \\
& = \la h_1^*(\delta' \cup \delta' ), [V] \ra =  \la \zeta \cup \zeta, [V] \ra = 0.
\end{aligned}
\end{equation*}
It follows (possibly~\cite{pontryagin:homotopy} is the original reference) that $h_1$ is homotopic to a map that misses a point of $\cptwo$, and hence is homotopic to a map $h: V \to S^2$ with the desired property.  Now the map $f_1 \times f_2 \times h$ is a degree one map that induces isomorphism on homology (and hence on cohomology).
\end{proof}

Now we consider the infinite cyclic covering space $\pp:V_\gamma\to V$ induced by the class $\gamma \in H^1(V)$.  The homology groups (resp.~mod $p$ homology groups) of $V_\gamma$ are modules over the ring $\Lambda = \Z[t,t^{-1}]$ (resp.~$\Lambda_p = \F_p[t,t^{-1}]$). Note that while $\Lambda_p$ is a PID, while $\Lambda$ is not. 
\begin{lemma}\label{L:Z-cover}
$V_\gamma$ satisfies $3$-dimensional Poincar\'e duality with coefficients in any field $\F$. Moreover, if $M\subset V$  is Poincar\'e dual to $\gamma$, then it lifts to $V_\gamma$ and represents the fundamental class of $V_\gamma$.
\end{lemma}
\begin{proof}
According to Milnor's Duality Theorem~\cite{milnor:covering} the conclusions of the theorem hold whenever $H_*(V_\gamma;\F)$ is finitely generated. (The fact that $M$ carries the fundamental class of $V_\gamma$ is implicit in Milnor's paper, and is proved explicitly in~\cite{kawauchi-matumoto:estimate}.)  Papadima-Suciu~\cite{papadima-suciu:spectral} and Pajitnov~\cite{pajitnov} (see also~\cite{farber:1-forms}) show that this finite generation holds if the {\em Aomoto complex}  
\begin{equation}\label{E:coho}
\xymatrix{
\cdots \ar[r] & H^j(V;\F) \ar[r]^{\cup \gamma} & H^{j+1}(V;\F) \ar[r]^{\cup \gamma} & H^{j+2}(V;\F) \ar[r] & \cdots
}
\end{equation}
is exact.  (Rather strikingly, an analytic version of this fact, which applies when $\F = \R$ or $\C$, is part of Taubes' theory of Fredholm complexes for manifolds with periodic ends~\cite[Theorem 3.1]{taubes:periodic}; compare~\cite{mrowka-ruberman-saveliev:derham}.) Exactness of \eqref{E:coho} for $\ts$ follows directly from the K\"unneth theorem; by hypothesis it holds for $V$ as well.
\end{proof}

The main result of this  is a version of Lemma~\ref{L:Z-cover} for some further $d$-fold covering spaces of $V_\gamma$ where $d=p^r$ and $p$ is a prime.  We continue (and slightly extend) the notation from above. 
Consider the covering space $V_{\eta_{d}} \to V$ where $\eta_{d}$ is the mod ${d}$ reduction of the generator $\eta \in H^1(V)$. This fits into a diagram of covering spaces where the vertical arrows are infinite cyclic coverings and the horizontal ones are $\Z_{p^r}$ coverings.
\[\xymatrixcolsep{6em}
\xymatrix{
V_{\gamma,\eta_{d}} \ar[r]^{p^r\text{-fold cover}} \ar[d]_{\Z-\text{cover}} & V_\gamma\ar[d]^{\Z-\text{cover}} \\
V_{\eta_{d}} \ar[r]^{p^r\text{-fold cover}} & V
}
\]
\begin{proposition}\label{P:prime-power torus}
The homology groups $H_*(V_{\gamma,\eta_{d}}; \C)$ are finite dimensional. In particular, $V_{\gamma,\eta_{d}}$ satisfies $3$-dimensional Poincar\'e duality with complex coefficients. 
\end{proposition}
Proposition~\ref{P:prime-power torus} will be proved by comparing $H_*(V_{\gamma,\eta_{d}}; \C)$ with $H_*(V_{\gamma,\eta_{d}}; \F_p)$, where $\F_p$ denotes the field with $p$ elements. 
The first step is to investigate the mod $p$ cohomology of $V_{\eta_{d}}$.
\begin{lemma}\label{L:prime torus}
$V_{\eta_{d}}$ is an $\F_{p}$ cohomology $\ts$.
\end{lemma}
Before embarking on the proof, it is worth considering what happens for the actual $\ts$. Recall the convention that $\gamma$ for a general $V$ corresponds to $\xi \in H^1(\ts)$. Then $(\ts)_\xi = \R \times S^1 \times S^2$ and so  $(\ts)_{\xi,\beta_d} = \R \times S^1 \times S^2$ has finitely generated homology. 
\begin{proof}
The infinite cyclic cover $V_\eta$ corresponding to $\eta$ factors through the cyclic $p^r$-fold cover $V_{\eta_{d}}\to V$
\begin{equation}\label{E:tower}
\xymatrix{
V_\eta \ar[rr]^{\Z-\text{cover}} \ar[dr]_{\Z-\text{cover}}&& V_{\eta_{d}} \ar[ld]^{p^r\text{-fold cover}}\\
&V & 
}
\end{equation}
If $t$ denotes the generator of the covering transformations of $V_\eta \to V$, then $t^{d}$ generates the covering transformations of $V_\eta \to V_{\eta_{d}}$.

The mod $p$ cohomology of $V_\eta$ is a finitely generated module over the PID $\Lambda_p = \F_p[t,t^{-1}]$ and hence splits as a finite sum of free $\Lambda_p$ modules and cyclic modules of the form $ \Lambda_p/\la q(t)\ra$. It is easy to see that the fact that $V$ is an $\F_p$ cohomology $\ts$ implies that the Aomoto complex \eqref{E:coho} is acyclic. According to~\cite[Proposition 9.4]{papadima-suciu:spectral}, this means that each homology group $H_k(V; \Lambda_p)$ is finite dimensional, and moreover any summand of the form 
$
\Lambda_p/((t-1)^j) 
$
has $j =1$.

To compute the homology of $V_{\eta_{d}}$, we examine the Milnor exact sequences with $\F_p$ coefficients for the two infinite cyclic covering spaces in \eqref{E:tower}, taking account of the relationship between their covering transformations. (To simplify the notation, the coefficient groups are not indicated.) The first reads
\begin{equation}\label{E:milnor}
\xymatrix@R-18pt{H_2(V) \ar[r] & H_1(V_\eta) \ar[r]^{t-1} & H_1(V_\eta) \ar[r]^{\pp_*} & H_1(V) \ar[r] &\\
&&  H_0(V_\eta) \ar[r]^{t-1} & H_0(V_\eta). &}
\end{equation}

Since $H_0(V_\eta) \cong \F_p$ and $t$ acts by the identity, the first sequence ends 
\[
\ldots\  H_1(V) = \F_p \oplus \F_p \to \F_p \to 0.
\]
In particular, we must have that $\coker [t-1: H_1(V_\eta) \to  H_1(V_\eta)] \cong \F_p$.

We remark that if the Laurent polynomial $q(t)$ is relatively prime to $t-1$, then $t-1$ is an isomorphism on $\Lambda_p/(q(t))$. We conclude that $H_1(V_\eta)$ must have exactly one summand of the form $\Lambda_p/(t-1) \cong \F_p$ and none of the form  $\Lambda_p/((t-1)^j)$ with $j>1$.

Now we look at the Milnor sequence relating the homology of $V_{\eta_{d}}$ to that of $V_\eta$.
\begin{equation}\label{E:milnorp}
\xymatrix@R-18pt{H_2(V_{\eta_{d}}) \ar[r] & H_1(V_\eta) \ar[r]^{t^{d}-1} & H_1(V_\eta) \ar[r]^{\pp_*} & H_1(V_{\eta_{d}}) \ar[r] &\\
&&  H_0(V_\eta) \ar[r]^{t^{d}-1} & H_0(V_\eta) &
}
\end{equation}
It is standard that (mod $p$, of course!) $t^{d} -1 = (t-1)^{d}$, which is an isomorphism on all but one summand in $H_1(V_\eta)$.  On the other hand, the cokernel of $ (t-1)^{d}$ on $\Lambda_p/(t-1)$ is again $\F_p$, and so the homology group $H_1(V_{\eta_{d}})$ is isomorphic to $\F_p \oplus \F_p$.

Since $V_{\eta_{d}}$ finitely covers $V_\eta$, which has Euler characteristic $0$, the Euler characteristic of $V_{\eta_{d}}$ must be $0$ as well. By Poincar\'e duality, we see that the mod $p$ homology groups of $V_{\eta_{d}}$ are the same as those of $\ts$. So it remains to be seen that the cup product $H^1(V_{\eta_{d}}) \times H^1(V_{\eta_{d}}) \to H^2(V_{\eta_{d}})$ is nonsingular. 

To this end, recall the homology equivalence $f:V \to \ts$ constructed in Lemma \ref{L:map}. It lifts to a degree-one map $f_{\eta_{d}}: V_{\eta_{d}} \to (\ts)_{\beta_{d}}$. Being a degree one map, $f_{\eta_{d}}$ induces a surjection in homology with any coefficients. But with $\F_p$ coefficients, the homology groups of $V_{\eta_{d}}$ and $(\ts)_{\beta_{d}}$ are isomorphic, so this surjection must in fact be an isomorphism. Dually, it is an isomorphism in cohomology, so by naturality, the cup product on $H^1(V_{\eta_{d}}) $ is non-singular.
\end{proof}

With these preliminaries in hand, we can now establish the finite-dimensionality of the homology of $V_{\gamma,\eta_{d}}$, which is a $\Z \times \Z_{d}$ cover of a cohomology $\ts$.
\begin{proof}[Proof of Proposition~\ref{P:prime-power torus}] By construction, $V_{\gamma,\eta_{d}}$ is an infinite cyclic cover of $V_{\eta_{d}}$, which is itself a $\Z_{d}$ cover of a cohomology $\ts$. We showed in Lemma~\ref{L:prime torus} that $V_{\eta_{d}}$ is an $\F_p$ cohomology torus, from which it follows that the cohomology of $V_{\gamma,\eta_{d}}$ with coefficients in $\F_p$ is finite dimensional. But by~\cite[Lemma 6]{casson-gordon:orsay} this implies that the homology of $V_{\gamma,\eta_{d}}$ with complex coefficients is also finite dimensional.  Milnor's Duality  Theorem~\cite{milnor:covering} shows that $V_{\gamma,\eta_{d}}$  satisfies Poincar\'e duality over $\C$.
\end{proof}

\subsection{The $\rho$ invariant of a cohomology $\ts$}\label{S:rho}  From the preceding discussion, we can now define $\rho$ invariants of a cohomology $\ts$ with distinguished generator $\gamma\in H^1(V)$.  We continue with the notation established in Section~\ref{S:notation}, and consider the character $\phi$ with $\phi(h) = \alpha$ and $\phi(g) = 1$. Suppose that $M$ is an oriented $3$-manifold embedded in $V$ that is Poincar\'e dual to $\gamma$.
\begin{proposition}\label{P:prime-power}
Suppose that $\alpha$ has prime-power-order. Then the quantity 
\begin{equation}\label{E:TL}
\rho_\alpha(V)= \rho_\alpha(M)
\end{equation}
is independent of the choice of $M$.
\end{proposition}
\begin{proof}
Let $M$ and $N$ be two such choices, and choose lifts of $M$ and $N$, say $M'$ and $N'$ to the infinite cyclic cover $V_\gamma \to V$ of $V$ corresponding to $\gamma$. Taking the image $N'$ under a high positive power of the covering transformation, we may assume that $M'$ and $N'$ are disjoint, and that $N'$ is to the right of $M'$. Composing $\phi$ with the maps induced by the projection $\pi$, we get that $\rho_{\phi}(M) = \rho_{\phi\circ \pi_*}(M')$, and similarly for $N$. 

Now $M'$ and $N'$ cobound a compact submanifold $W$ of $V_\gamma$, and $\phi\circ \pi_*$ extends to $\pi_1(W)$. So by definition (cf. \eqref{E:rho}) we know
\[
\rho_{\phi\circ \pi_*}(N') - \rho_{\phi\circ \pi_*}(M') = \sign(W) - \sign_{\phi\circ \pi_*}(W).
\]
We make a final appeal to Milnor's duality theorem. $M'$ carries the fundamental class of $V_\gamma$, thought of as a space satisfying $3$-dimensional Poincar\'e duality.  In other words, the inclusion map of $M'$ into $V_\gamma$ is a degree one map, and hence induces a surjection in rational homology. It follows that the intersection form on $H_2(V_\gamma;  \Q)$ vanishes. The same is true for $W$, and so $\sign(W) = 0$. Similarly, Proposition~\ref{P:prime-power torus} implies that the $p^r$ fold cover of $V_\gamma$ coming from $\phi$ has vanishing intersection form. It follows that the twisted signature $\sign_{\phi\circ \pi_*}(W)$ vanishes as well, so that 
\[
\rho_{\phi}(N) = \rho_{\phi\circ \pi_*}(N') = \rho_{\phi\circ \pi_*}(M') = \rho_{\phi}(M). 
\]
\end{proof}

In applications to knot concordance~\cite{litherland:torus} it is sometimes convenient to replace the \TL\  invariant by its average value around points on the circle where it jumps. We record a similar definition for the torus version of the invariant.
Choose a submanifold $M$ as above, and note that $\rho_{\phi}(M) $ is defined for arbitrary $\alpha \in S^1$. As a function on the circle, it is piecewise constant, with finitely many jumps at those characters $\phi_{\alpha_0}$ for which 
\[
\dim(H^1(M; \C_{\phi_{\alpha_0}}) > \lim_{\alpha \to \alpha_0^{-}} \dim(H^1(M; \C_{\phi_{\alpha}})\ \text{or}\ \lim_{\alpha \to \alpha_0^{+}} \dim(H^1(M; \C_{\phi_{\alpha}}).
\]
Informally, this means that the twisted cohomology jumps at $\alpha_0$; there are finitely many $\alpha_0 \in S^1$ where this happens. Let us normalize by defining $\bar\rho_{\phi}(M)$ to be the average value of the one-sided limits as $\beta \to \alpha^\pm$. In other words
\[
\bar\rho_{\phi}(M) = \frac12\left( \lim_{\beta \to \alpha^{-}} \rho_{\phi_\beta}(M) + \lim_{\beta \to \alpha^{+}} \rho_{\phi_\beta}(M)\right).
\]
It is again a piecewise constant function on $S^1$ with finitely many jumps.  

Since $\bar\rho$ is an average of integers, it is possible that it takes values in $\frac12\Z$, rather than in $\Z$. The following lemma shows that it is in fact an integer. 
\begin{lemma}\label{L:mod2} Let $Y$ be an oriented $3$-manifold and $\eta\in H^1(Y)$, and let $\phi_\alpha$ be the $\U(1)$ representation of $\pi_1(Y)$ given by the composition $\eta: H_1(Y) \to \Z \to \U(1)$ where $1 \in \Z$ is sent to $\alpha \in \U(1)$.  The  parity of the invariant $\rho_{\phi_\alpha}$ is constant off of a finite set. In particular, $\bar\rho_\alpha(Y)$ is an integer for all values of $\alpha$.
\end{lemma}
Chuck Livingston pointed out the idea used in the proof of diagonalizing over the ring of rational functions. 
\begin{proof}
Since the cobordism group $\Omega_3(\Z)$ vanishes, we can choose a $4$-manifold $W$ over which $\eta$ extends. After doing surgery on loops in $\ker[\pi_1(W) \to \Z]$ we may assume that in fact $\pi_1(W) \cong \Z$. It follows that all of the representations $\phi_\alpha$ extend to $W$, and so it suffices to show that the parity of the signatures $\sign_{\phi_\alpha}(W)$ that appear in the definition of $\rho_{\phi_\alpha}$ is constant on a dense set of $\alpha$. 

The local coefficient homology group $H_2(W;\Lambda)$ corresponding to the homomorphism $\eta: H_1(W) \to \Z$ supports the equivariant intersection form $q_\Lambda$. It is Hermitian with respect to the involution on $\Lambda$ given by $\bar{t} = t^{-1}$; see for instance~\cite{wall:book2,hambleton:intersection}.  For $\alpha \in \U(1)$, the form $q_\Lambda$ determines the intersection form on $H_2(W;\C_\alpha)$ as follows. An application of the K\"unneth spectral sequence as in~\cite{friedl-hambleton-melvin-teichner} shows that
$$
H_2(W;\C_\alpha) \cong H_2(W;\Lambda) \otimes_\Lambda \C
$$
where $t$ acts on $\C$ by multiplication by $\alpha$. This isomorphism respects intersections. 

To compute signatures, at least for generic values of $\alpha$ (to be specified momentarily) we pass to the quotient field $\Q(t)$ of $\Lambda$, and then tensor with $\R$ to get a form $q_{\R(t)}$. The form $q_{\R(t)}$ can then be diagonalized~\cite{gilmer-livingston:jumps,kearney:stable} in the following sense. With respect to a basis, represent $q_{\R(t)}$ by a matrix $B(t)$; then there is an invertible matrix $A(t)$ with $A B(t) A^* = D(t)$ a diagonal matrix.  Let the non-zero entries of $D(t)$ be rational functions $\lambda_1(t),\ldots,\lambda_k(t)$.

As long as $\alpha$ is not a pole of any $\lambda_j(t)$ or of any entry in $A(t)$, then $D(\alpha)$ is a representative for $q_\Lambda \otimes C_\alpha$. Hence (with finitely many exceptions) $\sign_\alpha(W)$ is the number of positive values of $\lambda_j(\alpha)$ minus the number of negative ones. With the exception of the (finitely many) zeros of the $\lambda_j(t)$, this is given by $k \pmod2$. The lemma follows.
\end{proof}
The arguments in~\cite{gilmer-livingston:jumps,kearney:stable} give more information about the location of the jumps in $\rho_{\phi_\alpha}$ and their relation to the Alexander polynomial in the special case when $Y$ is $0$-surgery on a knot. 

The extension of the $\rho$ invariant of $V$ to arbitrary characters is now a formal consequence of what has come before. 
\begin{definition}\label{D:TL}
Let $T$ be a torus embedded in $X$, a homology $\sss$. Let $V$ be the associated cohomology $\ts$ as before, and let $M$ be an oriented $3$-manifold dual to $\gamma$. Then for any $\alpha \in \U(1)$, define 
\begin{equation}\label{E:Vrho}
\rho_\alpha(V)  = \bar\rho_{\phi}(M).
\end{equation}
\end{definition}
The final result of this section is then
\begin{thm}\label{T:Vrho} Let $V$ be a cohomology $\ts$, and let $\{g,h\}$ be a basis for $H_1(V)$. Then for any $\alpha\in S^1$, the quantity $\rho_\alpha(V)$  in \eqref{E:Vrho} is well-defined.
\end{thm}
\begin{proof}
To see that $\rho_\alpha(V,\gamma)$ is well-defined, suppose that $M$ and $N$ are submanifolds of $V$ dual to $\gamma$. Then $\rho_{\phi}(M)$ and $\rho_{\phi}(N)$ are both well-defined piecewise continuous functions on $S^1$ with finitely many discontinuities. On the other hand, Proposition~\ref{P:prime-power} says that they agree on the dense subset of points of prime-power order. Hence the normalized versions $\bar\rho_{\phi}(M)$ and $\bar\rho_{\phi}(N)$ must agree at all points.
\end{proof}
It is worth remarking that $\rho_\alpha(V,\gamma)$, while independent of the choice of $3$-manifold $M$, may well depend on the choice of $\gamma$. 
\begin{example}\label{e:gamma}
Consider $V = S^1 \times S^3_0(K)$. Then for $\gamma$ pulled back from the circle, we will show below that  $\rho_\alpha(V,\gamma)$ is just the \TL\  invariant of $K$ and can well be non-zero.  On the other hand, if $\gamma$ is the pullback of the generator of $H^1(S^3_0(K))$, then the dual to $\gamma$ is represented by $S^1 \times F$ where $F$ is a Seifert surface for $K$, capped off in the surgered manifold.  But $S^1 \times F = \partial (S^1 \times H)$ where $H$ is a handlebody. Since  $H_2(S^1 \times F) \to H_2(S^1 \times H)$ is surjective, and the same is true for the twisted homology, both the ordinary and twisted signatures vanish. Hence $\rho_\alpha(V,\gamma) = 0$ for this choice of $\gamma$.
\end{example}

\section{The \TL\  invariant of a torus}\label{S:TL}
As described in the introduction, the idea is to construct, from a torus $T\subset X$ where $X$ is an oriented and homology oriented $\sss$, a cohomology $\ts$, say $V$. Then the \TL\  invariant of $(X,T)$ will be defined as the $\rho$ invariant of $V$ from the preceding section.  A small complication is that we need to see that the homology orientation gives rise to a particular class $\gamma \in H^1(V)$. As observed in Example~\ref{e:gamma} the $\rho$ invariant of $V$ may well depend on this choice.

In brief, we choose a 
diffeomorphism $f: T^2 \times \partial D^2 \to \partial\nu$. Then the manifold $V$ is defined to be  
\begin{equation}\label{E:Vglue}
V_{f} = T^2 \times \partial D^2 \cup_f X - \Int(\nu).
\end{equation}
We refer to this operation as a torus surgery. A useful model to keep in mind is $(X,T) = S^1 \times (Y,K)$ with $Y$ an integral homology sphere, and to have the gluing represent $S^1$ times an ordinary $0$-framed surgery on $K$.  
This model is a bit oversimplified, as there are many gluing maps that will produce a homology $\ts$.  Because many self-diffeomorphisms of a $3$-torus extend over $T^2 \times D^2$, many of these manifolds are diffeomorphic. They are presumably not all diffeomorphic, but it will turn out that they all give rise to the same $\rho$ invariants.

\subsection{Torus surgery}\label{S:gluing}
Surgery along a torus is a standard operation~\cite{fs:knots,baykur-sunukjian:round} in $4$-manifold topology, and is thoroughly discussed in in~\cite{plotnick:fibered}.  We start with a copy of $T^2 \times D^2$, whose boundary is  parameterized as a product of circles $a \times b \times m$ where $m$ is the boundary of $D^2$.  One bit of notation for this section: we will use the same letter for a circle and the homology class that it carries.

The torus $T \subset X$ may be parameterized as $g \times l$ where $g$ represents the preferred generator of $H_1(X)$, and $l$ is null-homologous in $X$. Note that $l$ is determined by this requirement, but $g$ is only well-defined up to replacing $g$ by $g+nl$ for an arbitrary integer $n$. Pick a framing of $\nu(T)$, which is specified by a section $\psi$ of the normal circle bundle. In particular, $\partial \nu(T)$ is now parameterized as $\psi(g) \times \psi(l) \times \mu_T$ where $T$ is the (oriented) meridian circle. Note that there is an action of $H^1(T)$ on the set of choices for $\psi$.  We can cut down this indeterminacy by choosing $\psi(l)$ to be null-homologous in the complement of $T$ (or in other words to have linking number $0$ with $T$). There does not seem to be a canonical choice for $\psi(g)$.

It is not hard to check, via the Mayer-Vietoris sequence, that $V$ is a cohomology $\ts$ if and only if $f(m) = \psi(l)$, and we choose one such gluing map. We refer to any such choice as a $0$-surgery on $X$ along $T$.  The next item to establish is that the cohomology class $\gamma\in H_1(X)$ naturally leads to a cohomology generator, also called $\gamma$, in $H^1(V)$.  This can also be done by analyzing the Mayer-Vietoris sequence, but we prefer the following more geometric argument.

Choose a connected oriented submanifold $Y^3$ of $X$ that is Poincar\'e dual to $\gamma$ and transverse to $T$.  It is helpful  to clean up the intersection.
\begin{lemma}\label{L:circle}
One can choose $Y$ so that $Y \cap T = K$, a knot with homology class $l$. 
\end{lemma}
\begin{proof}
The intersection $Y \cap T$ is a union of oriented circles in $T$, and Poincar\'e duality says that the sum of their homology classes is $l$. Hence $Y \cap T$ consists of a collection of circles that bound disks, together with a collection of $2n+1$ curves isotopic to $l$, with $n+1$ having the same orientation as  $l$ and $n$ oriented the opposite way. By a standard innermost disk/annulus argument, one can do surgery on $Y$ along the disks and an `annulus surgery' along pairs of oppositely oriented copies of $l$ to remove all but one curve.
\end{proof}
A related argument is given in~\cite[\S 2.4]{ma:surgery}.

Note that the normal bundle of $K$ in $Y$ is the restriction of the normal bundle of $T$ in $X$ to $K$. In particular, the (canonical) framing of $l$ induces a framing of $K$. Let $M$ be the result of surgery on $K \subset Y$ with that framing. Because of our choice of framing, the surgery on $Y$ can be done simultaneously (and inside of) the torus surgery on $X$. 

It follows that no matter what gluing $f$ we choose,  $M$ becomes a submanifold of $V$.  Now we can define the cohomology class $\gamma$ to be the Poincar\'e dual of $M$.  Because $M$ is connected and non-separating, it follows that $\gamma$ is a primitive class, and hence generates a summand of $H^1(V)$. Finally, we have our definition.
\begin{definition}\label{D:torusTL}
Let $X$ be an oriented homology $\sss$, and let $\gamma\in H^1(X)$ be a homology orientation. Let $V$ be any cohomology $\ts$ resulting from $0$-surgery along $T$, and let $\gamma \in H^1(V)$ be the cohomology generator described above. Then for $\alpha \in S^1$, define
\[
\sigma_\alpha(X,T,\gamma) = \rho_\alpha(V,\gamma).
\]
where $\rho_\alpha(V,\gamma)$ is defined in \eqref{E:Vrho}.
\end{definition}
In using this definition, it is important to recall that for $\alpha$ of prime-power order, $\rho_\alpha(V,\gamma)$ refers to the $\rho$ invariant of the $3$-manifold $M$, whereas for general $\alpha$, it refers to the average of the one-sided limits.

With these preliminary results in hand, we can establish the main theorem.
\begin{proof}[Proof of Theorem~\ref{T:TL}]
The discussion above shows that $\sigma_\alpha(X,T,\gamma) $ is well-defined, and we have already verified in Lemma~\ref{L:mod2} that it is an integer. So we need to verify the two properties stated in the theorem. 

The first property (the computation for a product $S^1 \times (Y,K)$ with $Y$ a homology sphere) is a consequence of an interpretation of the standard \TL\  invariant of $(Y,K)$ as a $\rho$ invariant. This seems to be standard for knots in the $3$-sphere--for $\alpha$ of finite order this is~\cite[Lemma 3.1]{casson-gordon:stanford} while for arbitrary $\alpha$ see for instance~\cite{friedl:eta,levine:eta,litherland:satellite,powell:4-genus}. The idea is that one can use the complement of a Seifert surface for $K$, pushed into $B^4$, to compute the $\rho$ invariant of $Y_0(X)$. The key computation relates the twisted signatures of this complement to the Seifert form of $K$. For an arbitrary homology sphere $Y$, much the same proof works, with $B^4$ replaced by an arbitrary compact $4$-manifold with boundary $Y$. We record the result here.
\begin{proposition}\label{P:eta-TL} For any $\alpha \in S^1$, we have $\sigma_\alpha(Y,K) = \rho_\alpha(Y_0(K))$.
\end{proposition}
Now consider a knot $K$ in a homology sphere $Y$, so that the product $(S^1 \times Y, S^1 \times K)$ is an embedded essential torus.  The cohomology class $\gamma$ pulled back from the standard generator of $H^1(S^1)$ provides a homology orientation.  Then we need to show that
\[
\sigma_\alpha(S^1 \times Y, S^1 \times K) = \sigma_\alpha(Y,K)
\]
Referring back to notation from Section~\ref{S:gluing}, the curve $l \subset T$ is just a copy of $K$, and its framing (restricted to $T \cap Y$) is exactly the usual $0$-framing of $K$ in $Y$. Hence the manifold $M$ in that section is just $Y_0(K)$, and the desired equation is just the definition of $\sigma_\alpha(S^1 \times Y, S^1 \times K)$ coupled with Proposition~\ref{P:eta-TL}.

The second property is straightforward from the definitions. Changing either orientation will reverse the orientation of the manifold $M$, and this will change the sign of $\rho_\alpha(M)$ for all prime-power $\alpha$. It follows that the averaged invariants $\bar\rho_\alpha(M)$ all change signs as well, which implies that  $\sigma_\alpha(X,T)$ changes sign.
\end{proof}
The fact that for $M$ given by $0$ surgery on a knot, the $\rho$ invariant used in our definition of $\sigma_\alpha(X,T)$ can be computed via the Seifert form of the knot has a generalization due to Neumann~\cite{neumann:signature}.  The key observation is that for any $M$ and $\alpha$, the representation $\phi_\alpha)$ factors through a homomorphism $H_1(M) \to \Z$.  Equivalently, $\phi_\alpha$ comes from a cohomology class $\eta \in H^1(M;\Z)$. The main result of~\cite{neumann:signature} then interprets $\rho_{\phi_\alpha}(M)$ in terms of an isometric structure (an algebraic construction using the cohomology of a surface dual to $\eta$.)  In the case of $0$ surgery, such an isometric structure is equivalent~\cite{kervaire:cobordism,levine:invariants} (up to a cobordism relation, respecting signatures) to the Seifert matrix that defines the \TL\ invariant.

\section{Examples and comparison with $\lfo(X,T,\alpha)$}\label{S:ex}
We compute the \TL\  invariant for three constructions of embedded tori.  The first two come from the mapping torus of the $n$-fold branched cover of a knot. In this mapping torus, we can consider the torus swept out by the fixed point set of the covering transformation, or alternatively the torus swept out by an invariant knot disjoint from the fixed point set. The third comes from taking a non-trivial  circle bundle over a $3$-manifold. 
Along with the product case, the first and last of these constructions were treated in Echeverria's paper~\cite{echeverria:tori}.  As verified more generally in~\cite{ma:signature}, the answers we get agree with those in~\cite{echeverria:tori}. 

\subsection{Branched covers--the branch set}\label{S:branched} The product case discussed in the proof of Theorem~\ref{T:TL} has a nice generalization. 
Consider a knot $K$ in a homology sphere $Y$. Then the $n$-twist spin~\cite{zeeman:twist} of $K$ is a fibered knot in a homology $4$-sphere.  (The description of twist-spinning in~\cite[\S 2]{litherland:deform} applies to a knot in an arbitrary homology sphere.) Surgery along this knot  produces a homology $\sss$ that we will call $X$. It is fibered with fiber $\Sigma$, the $n$-fold branched cover of $(Y,K)$, and the monodromy is a generator $\tau$ of the covering transformations. Let $J$ be the preimage of $K$ in $\Sigma$; it is the fixed point set of $\tau$ and hence sweeps out a torus $T$ in $X$.  The twist-spinning construction is not strictly necessary here; one could instead pass directly to the definition of $(X,T)$ as the (pairwise) mapping torus of $\tau$ on $(\Sigma,J)$. In~\cite{echeverria:tori} this construction was denoted $(X_\tau,T_\tau)$.

It seems natural to expect that $\sigma_\alpha(X,T)$ should be expressible in terms of the \TL\  invariants of $(Y,K)$, and we show that this is indeed the case by establishing a relation between the $\rho$ invariants of surgery on $Y$ along $K$ and those of surgery on $\Sigma$ along $J$.  The ideas here are well-known; compare~\cite{aps:II,casson-gordon:stanford}. Let $Y_0$ be the result of $0$-framed surgery along $K$, and let $\Sigma_0$ be its $n$-fold cyclic cover. It is standard that $H_1(\Sigma_0)$ splits as $\Z \oplus H_1(\Sigma)$, where the first summand is generated by the meridian of $J$.  Writing $d=p^r$, and $\alpha = e^{2 \pi i /d}$ then we want to compute $\sigma_{\alpha^k}(X,T)$. By definition, this is a $\rho$-invariant associated to a $\U(1)$ representation of $H_1(\Sigma_0)$ that vanishes on the second summand and takes the meridian of $J$ to $\alpha^k$.

Since $\Omega_3(\Z) = 0$, there is a $4$-manifold $W$ with $\partial W = Y_0$ and such that the inclusion map $H_1(Y_0) \to H_1(W)$ is an isomorphism. We will use $W$ and its covering spaces to relate the various $\rho$ invariants. Consider the tower of covering spaces 
\begin{equation*}
\xymatrix@C+4pc{
W_{dn} \ar[r]^{\Z_d-\text{cover}} \ar@/^2.5pc/@[red][rr]^{\Z_{dn}-\text{cover}}&W_{n} \ar[r]^{\Z_n-\text{cover}}  &W
}
\end{equation*}
and let $t$ denote the generator of the covering transformations of $W_{dn} \to W$. Then $t^n$ generates the covering transformations of $W_{dn} \to W_n$. Write $\omega =  e^{2 \pi i /dn}$, so that $\alpha = \omega^n$. Denote by $E(t,\omega^j)$ the $\omega^j$-eigenspace of $t$ acting on $H_2(W_{dn};\C)$; it is the same as the twisted cohomology $H_2(W;\C_\phi)$ where $\phi$ sends the generator of $H_1(W)$ to $\omega^{-j}$.  Using the symmetries of the \TL\  signature~\cite[Proposition 2.3]{conway:TL}, we have
\[
\sigma_{\omega^j}(Y,K) = \sigma_{\omega^{-j}}(Y,K) = \sign(W) - \sign(E(t,\omega^j))
\]
where the last term means the signature of the Hermitian intersection form on $E(t,\omega^j)$.   There is a similar decomposition of $H_2(W_{dn};\C)$ into eigenspaces $E(t^n,\omega^{nk})$ for the action of $t^n$.
\begin{thm}\label{T:twist}
For any $k$, the invariant $\sigma_{\alpha^k}(X,T)$ is a sum of \TL\  invariants of $K$:
\begin{equation}\label{E:spun TL}
\sigma_{\alpha^k}(X,T) = \sigma_{\alpha^k}(\Sigma,J) = -\sum_{j=1}^{n-1} \sigma_{\omega^{dj}}(Y,K) + \sum_{j=0}^{n-1}\sigma_{\omega^{dj+k}}(Y,K).
\end{equation}
\end{thm}
\begin{proof}
The main point is the decomposition
\begin{equation}\label{E:decomp}
E(t^n,\omega^{nk}) = \bigoplus_{j=0}^{n-1}E(t,\omega^{dj+k})
\end{equation}
of $E(t^n,\omega^{nk})$ into eigenspaces of $t$.  Since the different eigenspaces are orthogonal with respect to the intersection form, this translates into an identity on signatures. The definition of the \TL\  signature and the decomposition \eqref{E:decomp} thus give
\[
\begin{aligned}
\sign(E(t,\omega^{dj+k})) &= \sign(W) - \sigma_{\omega^{dj+k}}(Y,K)\\
 \sigma_{\alpha^k}(\Sigma,J) & = \sign(W_n) -\sign(E(t^n,\omega^{nk}))\\
 &=  \sign(W_n) - n \sign(W) + \sum_{j=0}^{n-1}\sigma_{\omega^{dj+k}}(Y,K).
\end{aligned}
\]
Similarly, the decomposition of $H_2(W_n;\C)$ into eigenspaces gives that 
\[ 
\begin{aligned}
\sign(W_n) & =  \sign(W) + \sum_{j=1}^{n-1} j^{th}\ \text{eigenspace signature for }W_n\\
& = n \sign(W) -  \sum_{j=1}^{n-1}\sigma_{\omega^{dj}}(Y,K).
\end{aligned}
\]
Putting these together (note that the signature of $W$ cancels, as expected) gives the statement of the theorem.
\end{proof}
Comparing this calculation with~\cite[Theorem 44]{echeverria:tori}, and accounting for the somewhat different notation, we  see that the expected relation \eqref{E:conj} holds. 
\subsection{Branched covers--periodic knots}\label{S:periodic}
To carry out the mapping torus construction above, all we needed was knot in $\Sigma$ that was invariant under a suitable transformation $\tau$. Here is a more general version. Suppose that $L=(K,J)$ is an oriented link, and let $\pi: \Sigma \to S^3$ be the $n$-fold cover of $S^3$ branched along $K$, with $\tau$ the generator of group of covering transformations. Assuming that $n$ is relatively prime to the linking number of $J$ and $K$, the preimage $\jtilde = \pi^{-1}(J)$ is a knot in $\Sigma$ that is invariant under $\tau$.
\begin{definition}
The torus $T_{L,n}$ is given by the mapping torus 
\[
S^1 \times_\tau \jtilde \subset S^1 \times_\tau \Sigma = X.
\]
\end{definition}
By construction, $T_{L,n}$ is essential, and we calculate its \TL\ invariant. To compute $\sigma_{\alpha}(X,T_{L,n})$ for arbitrary $\alpha \in S^1$, it suffices to compute it for all $\alpha$ of finite (indeed prime-power) order. 
The answer is similar to the formula in Theorem~\ref{T:twist} and is expressed in terms of an extension to links of the \TL\ invariant due to Cooper~\cite{cooper:thesis,cooper:abelian-cover}. This invariant associates to a $2$-component oriented link, and a pair of unit complex numbers $(\omega_1,\omega_2) \in (S^1 - \{1\}) \times (S^1 - \{1\})$ a signature $\sigma_L(\omega_1,\omega_2)$. The scope of $\sigma_L$  was extended by Cimasoni-Florens~\cite{cimasoni-florens:signatures} to include invariants of `colored links', with one $\omega$ for each color. Let us adopt the convention that $\sigma_L(\omega_1,1)$ is the usual \TL\ signature of the first component, and similarly for $\sigma_L(1,\omega_2)$.   
\begin{thm}\label{T:periodic}
If $\alpha \in S^1$ has prime-power order, then the invariant $\sigma_{\alpha^k}(X,T)$ is a sum of \TL\  invariants of $L$:
\begin{equation}\label{E:periodic TL}
\sigma_{\alpha}(X,T_{L,n}) = \sum_{j=0}^{n-1} \sigma_{L}(e^{2\pi i j/n},\alpha) - \sum_{j=1}^{n-1} \sigma_{L}(e^{2\pi i j/n},1).
\end{equation}
\end{thm}
Combined with Ma's theorem~\cite{ma:signature} that we get a calculation of Echeverria's invariant $\lfo(X,T,\alpha)$ in this case. Again the answer is in terms of \TL\ invariants of the link. 
\begin{corollary}\label{C:periodic}
For $\alpha \in S^1$ of prime-power order $d$,
\[
\lfo(X,T_{L,n},\alpha)  = \sum_{j=0}^{n-1} \sigma_{L}(e^{2\pi i j/n},\alpha^2).
\]  
\end{corollary}
To prove the corollary, recall that Ma showed that $\lfo(X,T,\alpha) = 8\lfo(X) + \sigma_{\alpha^2}(X,T)$. On the other hand, the author and Saveliev~\cite{ruberman-saveliev:mappingtori} computed 
\[
\lfo(X) = \frac18 \sum_{j=1}^{n-1} \sigma_{L}(e^{2\pi i j/n},1)
\]
which cancels the second term in \eqref{E:periodic TL}.
\begin{proof}[Proof of Theorem~\ref{T:periodic}]
The proof is similar in spirit to that of Theorem~\ref{T:twist}, with some additional signature calculations. The details are basically an adaptation of well-known arguments, notably the `$4$-dimensional interpretation' of $\sigma_L(1,\omega_2)$ as discussed in in~\cite[\S 6]{cimasoni-florens:signatures}, and so we will be a bit sketchy in our treatment. To simplify notation a little, write $T$ for the torus $T_{L,n}$.

Choose a pair of transversally intersecting surfaces $F_K$ and $F_J$ in $B^4$ with boundary $L$.  The $n$-fold cyclic branched covering $\pi: \Sigma \to S^3$ over $K$ extends to branched cover $\pi:W_n \to B^4$ branched over $F_K$.  Let $\jtilde$ (resp. $\ftilde_J$) denote the preimage of $J$ (resp. $F_J$). Denote by $\Sigma_0(\jtilde)$ the $0$-framed surgery along $\jtilde$ in $\Sigma$, where the $0$-framing is the one that extends over $\ftilde_J$. Note that $\Sigma_0(\jtilde)$ embeds in $V$, the result of $0$-surgery on $T$, and is Poincar\'e dual to $\gamma \in H^1(V)$, so that by definition,
\[
\sigma_\alpha(X,T) = \rho_{\phi_\alpha}(\Sigma_0(\jtilde)).
\]

In principle, to compute this one would need a $4$-manifold over which the $d$-fold cyclic covering of $\Sigma_0(\jtilde)$ corresponding to the character $\phi_\alpha$ extends. But a now-standard argument~\cite[\S 2]{casson-gordon:stanford} shows that it can be computed instead by extending to a branched covering. In addition to the equivariant signatures of the branched cover, there is a term proportional to the self-intersection of the branch set. Let $W_n^0$ denote the result of adding a $0$-framed $2$-handle to $W_n$ along $\jtilde$.  Then the covering corresponding to $\phi_\alpha$ extends to a branched covering $W_{nd}^0 \to W_n^0$ with branch set equal to the union of $\ftilde_J$ with the core of the $2$-handle. 

By construction, the self-intersection of the branch set is $0$. Hence $\rho_{\phi_\alpha}(\Sigma_0(\jtilde))$ can be computed as the difference between the signature of $W_n^0$ and an appropriate equivariant signature of $W_{nd}^0$. It is straightforward to prove (from the choice of framing being $0$) that the signature of $W_n^0$ is the same as the signature of $W_n$, and that the same is true for the equivariant signatures of $W_{nd}^0$ and $W_{nd}$, the $d$-fold cover of $W$ branched along $\ftilde_J$.

Decomposing the homology of $W_n$ into eigenspaces for the action of the covering transformation from the branched covering $W_n \to B^4$ gives that 
\[
\sign(W_n)  = \sum_{j=1}^{n-1} \sigma_{L}(e^{2\pi i j/n},1).
\]

Similarly, $W_{nd}$ is an $nd$-fold branched cover over $B^4$ with branch set $F_K \cup F_J$, and so its homology has an action of $\Z_n \times \Z_d$. In particular, each eigenspace for the $\Z_d$ action is preserved by the $\Z_n$ action, and so has a further eigenspace decomposition. It follows that the equivariant signature of $W_{nd}$ is given by
\[
\sum_{j=0}^{n-1} \sigma_{L}(e^{2\pi i j/n},\alpha) 
\]
and the theorem follows.
\end{proof}

\subsection{Circle bundles}\label{S:bundle} Some time ago, Baldridge suggested to the author and Saveliev the following construction of a homology $\sss$, which was discussed in~\cite{ruberman-saveliev:survey}. Let $N$ be a homology $S^1 \times S^2$, with a generator $\gamma_{N} \in H^1(N)$. Let $\pi:X \to N$ be the oriented circle bundle  whose Euler class evaluates to $1$ on the Poincar\'e dual of $\gamma_Y$. Then $X$ is a homology $\sss$, and $\gamma= \pi^*\gamma_N$ is a homology orientation.  In order for $\lfo(X)$ to be defined, $N$ must have the $\Z[\Z]$ homology of $S^1 \times S^2$, and that restriction implies~\cite[\S 8]{ruberman-saveliev:survey} that in fact $\lfo(X) = 0$.

In this same situation, Echeverria constructs a torus embedded in $X$ and conjectures (see Conjecture 51) that for $\alpha$ for which it is defined, his invariant is $0$. Here is a description that is equivalent to his, but easier for our purposes. Consider an oriented knot $J$ in $N$ such that $\la \gamma_N, J\ra = 1$. Then let $T = \pi^{-1}(J)$. This of course depends on the precise choice of $J$, but the calculation does not. 
\begin{proposition}\label{P:bundle} For any $\alpha \in S^1$, the invariant $\sigma_\alpha(X,T)$ vanishes.
\end{proposition}
\begin{proof}
To compute $\sigma_\alpha(X,T)$, we need a submanifold $M$ dual to $\gamma$.  Choose an oriented surface $F \subset N$ that is Poincar\'e dual to $\gamma_N$, and let $M = \pi^{-1}(N)$. It is readily verified that $M$ is then Poincar\'e dual to $\gamma$.  

By construction, $M$ is the circle bundle over $F$ with Euler class $1$, and the Gysin sequence shows that its fibers are null-homologous in $N$ (and of course in $X$ as well). In particular, the fiber of $\pi:T \to J$ is the curve $l$ appearing in the definition of $\sigma_\alpha(X,T)$. When we do a $0$-surgery on $T$, the resulting surgery on $M$ is therefore the $0$-framed surgery along $J$. The result of this surgery is just $S^1 \times F$. This may be verified by careful thinking about the clutching function for the bundle $\pi:M \to F$ but some readers may prefer the handle picture presented momentarily. In either event, the $\rho$ invariants for $S^1 \times F$ all vanish, because any representation $\phi_\alpha$ will extend over $S^1 \times H$ where $H$ is a handlebody.

To verify that surgery on $M$ is a product, consider the following standard picture~\cite{gompf-stipsicz:book} of the Euler class $1$ circle bundle over $F$. If F has genus $g$, then the pattern is repeated $g$ times; the red curve at the right of the picture is the fiber.  
\begin{figure}[h]
\labellist
\small\hair 2pt
\pinlabel {$1$} [ ] at  168 48
\pinlabel {$0$} [ ] at  40 94
\pinlabel {$0$} [ ] at  30 54 
\pinlabel {$0$} [ ] at 106 54
\pinlabel {$0$} [ ] at 113 94
\pinlabel {$g$} [ ] at 85 80
\pinlabel {{\color{red}{$l$}}} [ ] at 175 20
\endlabellist
    \centering
        \includegraphics[scale=1.8]{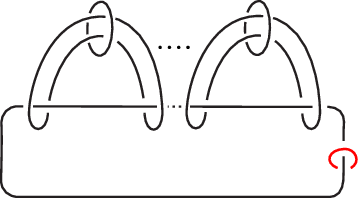}
        \caption{Fiber $l$ in the Euler class $1$ circle bundle over $F_g$}
        \label{F:fiber}
\end{figure}
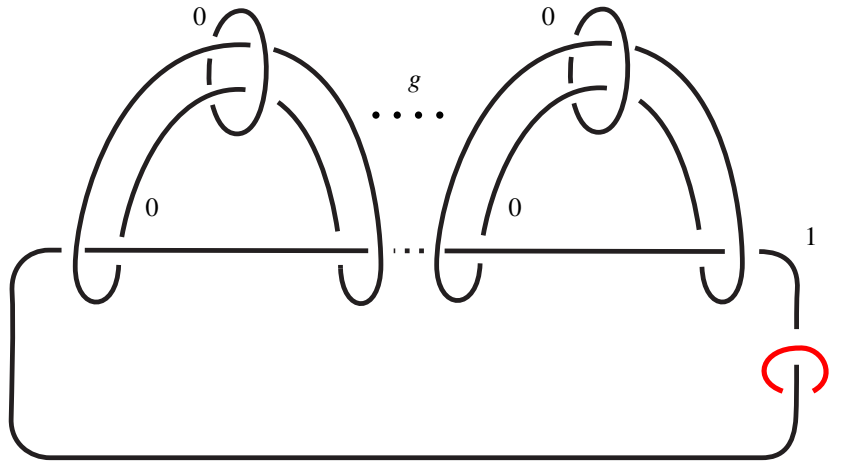
The $0$ framed pushoff of $l$ is the one that has blackboard framing $1$ in the picture, so we are doing surgery on $l$ with coefficient $1$. Blow down the red curve,  changing the framing of the curve that it links to $0$, and  the figure becomes the standard picture of $S^1 \times F$. 
\end{proof}

\bibliography{torusTL}

\providecommand{\bysame}{\leavevmode\hbox to3em{\hrulefill}\thinspace}
\begin{thebibliography}{10}

\bibitem{akbulut-mccarthy}
S.~Akbulut and J.~McCarthy, \emph{Casson's invariant for oriented homology
  3-spheres -- an exposition}, Mathematical Notes, vol.~36, Princeton
  University Press, Princeton, 1990.

\bibitem{aps:II}
M.F. Atiyah, V.K. Patodi, and I.M. Singer, \emph{Spectral asymmetry and
  {Riemannian} geometry: {II}}, Math.\ Proc.\ Camb.\ Phil.\ Soc. \textbf{78}
  (1975), 405--432.

\bibitem{atiyah-singer:III}
M.F. Atiyah and I.M. Singer, \emph{The index of elliptic operators: {III}},
  Ann. of Math. \textbf{87} (1968), 546--604.

\bibitem{baykur-sunukjian:round}
R.~{\.I}nan{\c{c}} Baykur and Nathan Sunukjian, \emph{Round handles,
  logarithmic transforms and smooth 4-manifolds}, J. Topol. \textbf{6} (2013),
  no.~1, 49--63.

\bibitem{casson-gordon:stanford}
A.~Casson and C.~Gordon, \emph{On slice knots in dimension three}, Proc.\
  Symp.\ Pure Math. \textbf{32} (1978), 39--53.

\bibitem{casson-gordon:orsay}
\bysame, \emph{Cobordism of classical knots}, {\`{A}} la Recherche de la
  Topologie Perdue (A.~Marin and L.~Guillou, eds.), Progress in Mathematics,
  Birkhauser, Boston, 1986, pp.~181--199.

\bibitem{cimasoni-florens:signatures}
David Cimasoni and Vincent Florens, \emph{Generalized {S}eifert surfaces and
  signatures of colored links}, Trans. Amer. Math. Soc. \textbf{360} (2008),
  no.~3, 1223--1264.

\bibitem{conway:TL}
Anthony Conway, \emph{The {Levine--Tristram} signature: a survey}, 2019,
  \url{http://arxiv.org/abs/1903.04477}.

\bibitem{cooper:thesis}
D.~Cooper, \emph{Signatures of surfaces in 3-manifolds and applications to knot
  and link cobordism}, Ph.D. thesis, Warwick University, 1982.

\bibitem{cooper:abelian-cover}
\bysame, \emph{The universal abelian cover of a link}, Low-dimensional topology
  ({B}angor, 1979), London Math. Soc. Lecture Note Ser., vol.~48, Cambridge
  Univ. Press, Cambridge, 1982, pp.~51--66.

\bibitem{echeverria:tori}
Mariano Echeverria, \emph{A generalization of the {Tristram-Levine} knot
  signatures as a singular {Furuta-Ohta} invariant for tori}, 2019,
  \url{http://arxiv.org/abs/1908.11359}.

\bibitem{farber:1-forms}
Michael Farber, \emph{Topology of closed one-forms}, Mathematical Surveys and
  Monographs, vol. 108, American Mathematical Society, Providence, RI, 2004.

\bibitem{fs:knots}
Ronald Fintushel and Ronald~J. Stern, \emph{Knots, links, and $4$-manifolds},
  Invent. Math. \textbf{134} (1998), no.~2, 363--400.

\bibitem{friedl:eta}
Stefan Friedl, \emph{Eta invariants as sliceness obstructions and their
  relation to {C}asson-{G}ordon invariants}, Algebr. Geom. Topol. \textbf{4}
  (2004), 893--934.

\bibitem{friedl-hambleton-melvin-teichner}
Stefan Friedl, Ian Hambleton, Paul Melvin, and Peter Teichner,
  \emph{Non-smoothable four-manifolds with infinite cyclic fundamental group},
  Int. Math. Res. Not. IMRN (2007), no.~11, Art. ID rnm031, 20.

\bibitem{furuta-ohta}
Mikio Furuta and Hiroshi Ohta, \emph{Differentiable structures on punctured
  $4$-manifolds}, Topology Appl. \textbf{51} (1993), no.~3, 291--301.

\bibitem{gilmer-livingston:jumps}
Patrick~M. Gilmer and Charles Livingston, \emph{Signature jumps and {A}lexander
  polynomials for links}, Proc. Amer. Math. Soc. \textbf{144} (2016), no.~12,
  5407--5417.

\bibitem{gompf-stipsicz:book}
Robert~E. Gompf and Andr{\'a}s~I. Stipsicz, \emph{$4$-manifolds and {K}irby
  calculus}, American Mathematical Society, Providence, RI, 1999.

\bibitem{hambleton:intersection}
Ian Hambleton, \emph{Intersection forms, fundamental groups and 4-manifolds},
  Proceedings of {G}\"{o}kova {G}eometry-{T}opology {C}onference 2008,
  G\"{o}kova Geometry/Topology Conference (GGT), G\"{o}kova, 2009,
  pp.~137--150.

\bibitem{herald:alexander}
Christopher~M. Herald, \emph{Flat connections, the {A}lexander invariant, and
  {C}asson's invariant}, Comm. Anal. Geom. \textbf{5} (1997), no.~1, 93--120.

\bibitem{heusener:curves}
Michael Heusener, \emph{{${\SO}_3({\R})$}-representation curves for two-bridge
  knot groups}, Math. Ann. \textbf{298} (1994), no.~2, 327--348.

\bibitem{heusener-kroll}
Michael Heusener and Jochen Kroll, \emph{Deforming abelian
  {${\SU}(2)$}-representations of knot groups}, Comment. Math. Helv.
  \textbf{73} (1998), no.~3, 480--498.

\bibitem{kauffman-taylor:links}
L.H. Kauffman and L.R. Taylor, \emph{Signature of links}, Trans.\ A.M.S.
  \textbf{216} (1976), 351--365.

\bibitem{kawauchi-matumoto:estimate}
Akio Kawauchi and Takao Matumoto, \emph{An estimate of infinite cyclic
  coverings and knot theory}, Pacific J. Math. \textbf{90} (1980), no.~1,
  99--103.

\bibitem{kearney:stable}
M.~Kate Kearney, \emph{The stable concordance genus}, New York J. Math.
  \textbf{20} (2014), 973--987.

\bibitem{kervaire:cobordism}
Michel~A. Kervaire, \emph{Knot cobordism in codimension two},
  Manifolds--{A}msterdam 1970 ({P}roc. {N}uffic {S}ummer {S}chool), Lecture
  Notes in Mathematics, Vol. 197, Springer, Berlin, 1971, pp.~83--105.

\bibitem{kronheimer-mrowka:I}
P.B. Kronheimer and T.S. Mrowka, \emph{Gauge theory for embedded surfaces {I}},
  Topology \textbf{32} (1993), 773--826.

\bibitem{kronheimer-mrowka:II}
\bysame, \emph{Gauge theory for embedded surfaces {II}}, Topology \textbf{34}
  (1995), 37--97.

\bibitem{levine:invariants}
J.~Levine, \emph{Invariants of knot cobordism}, Inventiones Math. \textbf{8}
  (1969), 98--110.

\bibitem{levine:cobordism}
\bysame, \emph{Knot cobordism groups in codimension two}, Comment. Math. Helv.
  \textbf{44} (1969), 229--244.

\bibitem{levine:eta}
J.~P. Levine, \emph{Link invariants via the eta invariant}, Comment. Math.
  Helv. \textbf{69} (1994), no.~1, 82--119.

\bibitem{lin:rep}
Xiao-Song Lin, \emph{A knot invariant via representation spaces}, J.
  Differential Geom. \textbf{35} (1992), no.~2, 337--357.

\bibitem{litherland:satellite}
R.~Litherland, \emph{Cobordism of satellite knots}, Four-Manifold Theory
  (C.~Gordon and R.~Kirby, eds.), American Math.\ Soc, Providence, 1984,
  pp.~327--334.

\bibitem{litherland:deform}
R.~A. Litherland, \emph{Deforming twist-spun knots}, Trans. Amer. Math. Soc.
  \textbf{250} (1979), 311--331.

\bibitem{litherland:torus}
\bysame, \emph{Signatures of iterated torus knots}, Topology of low-dimensional
  manifolds ({P}roc. {S}econd {S}ussex {C}onf., {C}helwood {G}ate, 1977),
  Lecture Notes in Math., vol. 722, Springer, Berlin, 1979, pp.~71--84.

\bibitem{ma:surgery}
Langte Ma, \emph{A surgery formula for the {Casson-Seiberg-Witten} invariant of
  integral homology {$S^1 \times S^3$}}, 2019,
  \url{https://arxiv.org/abs/1909.01533}.

\bibitem{ma:signature}
\bysame, \emph{Equivariant torus signature and periodic rho invariant}, 2021,
  \url{https://arxiv.org/abs/2101.10243}.

\bibitem{milnor:covering}
J.~Milnor, \emph{Infinite cyclic coverings}, Topology of Manifolds (J.~Hocking,
  ed.), Prindle, Weber and Schmidt, Boston, 1968, pp.~115--133.

\bibitem{mrowka-ruberman-saveliev:derham}
Tomasz Mrowka, Daniel Ruberman, and Nikolai Saveliev, \emph{Index theory of the
  de {R}ham complex on manifolds with periodic ends}, Algebr. Geom. Topol.
  \textbf{14} (2014), no.~6, 3689--3700.

\bibitem{neumann:signature}
Walter~D. Neumann, \emph{Signature related invariants of manifolds. {I}.
  {M}onodromy and {$\gamma $}-invariants}, Topology \textbf{18} (1979), no.~2,
  147--172.

\bibitem{papadima-suciu:spectral}
Stefan Papadima and Alexander~I. Suciu, \emph{The spectral sequence of an
  equivariant chain complex and homology with local coefficients}, Trans. Amer.
  Math. Soc. \textbf{362} (2010), no.~5, 2685--2721.

\bibitem{pajitnov}
A.~V. Pazhitnov, \emph{Proof of a conjecture of {N}ovikov on homology with
  local coefficients over a field of finite characteristic}, Dokl. Akad. Nauk
  SSSR \textbf{300} (1988), no.~6, 1316--1320.

\bibitem{plotnick:fibered}
Steven~P. Plotnick, \emph{Fibered knots in {$S\sp 4$}---twisting, spinning,
  rolling, surgery, and branching}, Four-manifold theory (Durham, N.H., 1982),
  Contemp. Math., vol.~35, Amer. Math. Soc., Providence, RI, 1984,
  pp.~437--459.

\bibitem{pontryagin:homotopy}
L.~S. Pontryagin, \emph{On a connection between homology and homotopy},
  Izvestiya Akad. Nauk SSSR. Ser. Mat. \textbf{13} (1949), 193--200.

\bibitem{powell:4-genus}
Mark Powell, \emph{The four-genus of a link, {L}evine-{T}ristram signatures and
  satellites}, J. Knot Theory Ramifications \textbf{26} (2017), no.~2, 1740008,
  28.

\bibitem{ruberman:ds}
Daniel Ruberman, \emph{Doubly slice knots and the {C}asson-{G}ordon
  invariants}, Trans. Amer. Math. Soc. \textbf{279} (1983), no.~2, 569--588.

\bibitem{ruberman:ds2}
\bysame, \emph{The {C}asson-{G}ordon invariants in high-dimensional knot
  theory}, Trans. Amer. Math. Soc. \textbf{306} (1988), no.~2, 579--595.

\bibitem{ruberman-saveliev:mappingtori}
Daniel Ruberman and Nikolai Saveliev, \emph{Rohlin's invariant and gauge
  theory. {II}. {M}apping tori}, Geom. Topol. \textbf{8} (2004), 35--76
  (electronic).

\bibitem{ruberman-saveliev:survey}
\bysame, \emph{{C}asson--type invariants in dimension four}, {G}eometry and
  {T}opology of {M}anifolds, Fields Institute Communications, vol.~47, AMS,
  2005, pp.~281--306.

\bibitem{saveliev:casson}
Nikolai Saveliev, \emph{Lectures on the topology of 3-manifolds}, revised ed.,
  de Gruyter Textbook, Walter de Gruyter \& Co., Berlin, 2012, An introduction
  to the Casson invariant.

\bibitem{taubes:periodic}
C.~Taubes, \emph{Gauge theory on asymptotically periodic 4-manifolds}, J.
  Diff.\ Geo. \textbf{25} (1987), 363--430.

\bibitem{taubes:casson}
\bysame, \emph{Casson's invariant and gauge theory}, J. Diff.\ Geo. \textbf{31}
  (1990), 547--599.

\bibitem{tristram}
A.~G. Tristram, \emph{Some cobordism invariants for links}, Proc. Cambridge
  Philos. Soc. \textbf{66} (1969), 251--264.

\bibitem{viro:links}
O.~Ja. Viro, \emph{Branched coverings of manifolds with boundary, and
  invariants of links. {I}}, Math. USSR-Izvestia \textbf{7} (1975), 1239--1256.

\bibitem{wall:book2}
C.~T.~C. Wall, \emph{Surgery on compact manifolds}, second ed., Mathematical
  Surveys and Monographs, vol.~69, American Mathematical Society, Providence,
  RI, 1999, Edited and with a foreword by A. A. Ranicki.

\bibitem{zeeman:twist}
E.~C. Zeeman, \emph{Twisting spun knots}, Trans. Amer. Math. Soc. \textbf{115}
  (1965), 471--495.

\end{thebibliography}
\bibliographystyle{amsplain}

\end{document}